\documentclass[letterpaper,11pt,reqno]{amsart} 
\usepackage[portrait,margin=3.3cm]{geometry}
\usepackage{systeme}
\usepackage{mathrsfs,xfrac} 
\usepackage[colorlinks=true,linkcolor=blue,citecolor=blue,urlcolor=blue]{hyperref} 
\usepackage{amsmath,amssymb,amsthm,amsfonts,amsbsy,latexsym,dsfont,color} 
\usepackage[numeric,initials,nobysame]{amsrefs} 
\usepackage[foot]{amsaddr}

\usepackage[utf8]{inputenc}
\usepackage[english]{babel}
\usepackage{comment}

\usepackage[textsize=tiny]{todonotes}
\usepackage{regexpatch}
\makeatletter
\xpatchcmd{\@todo}{\setkeys{todonotes}{#1}}{\setkeys{todonotes}{inline,#1}}{}{}
\makeatother

\usepackage{enumerate}

\newtheorem{thm}{Theorem}[section]
\newtheorem{lem}[thm]{Lemma}
\newtheorem{cor}[thm]{Corollary}

\newtheorem{ex}[thm]{Example}
   
\renewcommand{\le}{\leqslant} 
\renewcommand{\ge}{\geqslant}

\newcommand{\ra}{\rangle}
\newcommand{\la}{\langle}

\newcommand{\ind}{\mathds{1}}
\newcommand{\eps}{\varepsilon}

\newcommand{\norm}[1]{\left\Vert#1\right\Vert}
\newcommand{\abs}[1]{\left\vert#1\right\vert}

\newcommand{\ie}{\emph{i.e.,}}

\let\ga=\alpha \let\gb=\beta   
     \let\gl=\lambda       \let\gn=\nu    \let\gs=\sigma  
  
 \let\gD=\Delta  \let\gL=\Lambda 
         \let\gS=\Sigma  
                                
\newcommand{\cC}{\mathcal{C}}
\newcommand{\cE}{\mathcal{E}}

\newcommand{\cP}{\mathcal{P}}

\newcommand{\cV}{\mathcal{V}}
  
\newcommand{\vR}{\mathbf{R}}\newcommand{\vS}{\mathbf{S}}

\newcommand{\vx}{\mathbf{x}}
 
\newcommand{\mv}[1]{\boldsymbol{#1}}

\newcommand{\mvo}{\boldsymbol{o}}
\newcommand{\mvq}{\boldsymbol{q}}
\newcommand{\mvu}{\boldsymbol{u}}\newcommand{\mvv}{\boldsymbol{v}}
\newcommand{\mvx}{\boldsymbol{x}}

\newcommand{\mvgs}{\boldsymbol{\sigma}}

\newcommand{\bN}{\mathbb{N}}
\newcommand{\bR}{\mathbb{R}}

\newcommand{\sS}{\mathscr{S}}

\DeclareMathOperator{\E}{\mathds{E}}

\DeclareMathOperator{\diag}{diag}

\DeclareMathOperator{\N}{N}

\DeclareMathOperator{\av}{Av}

\newcommand{\pd}{\emph{positive semi-definite}}

\newcommand{\half}{\sfrac12}

\newcommand{\T}{\intercal}

\begin{document}

\title[TAP equations for MSK model]{Thouless-Anderson-Palmer equations for the Multi-species Sherrington-Kirkpatrick model}

\author[Wu]{Qiang Wu$^\dagger$}
\address{Department of Mathematics, University of Illinois at Urbana-Champaign, 1409 W Green Street, Urbana, Illinois 61801}
\email{ $^\dagger$qiangwu2@illinois.edu}
\date{\today}
\begin{abstract}
	We prove the Thouless-Anderson-Palmer (TAP) equations for the local magnetization in the multi-species Sherrington-Kirkpatrick (MSK) spin glass model. One of the key ingredients is based on concentration results established in~\cite{DW21a}. The equations hold at high temperature for general MSK model without \pd~assumption on the variance profile matrix $\mathbf{\gD}^2$. 
\end{abstract}

\maketitle

\section{Introduction and Main results}

The multi-species Sherrington-Kirkpatrick (MSK) model~\cite{BCMT15} is a non-homogeneous variant of the classical Sherrington-Kirkpatrick (SK) spin glass model~\cite{SK72}. In the SK model, the interactions among spins are characterized by a family of i.i.d. random variables. While in the MSK model, the spins are divided into different types, the interactions inside a particular type and among different types are now different. Due to this multi-type structure, the MSK model has also been used to study the behavior of systems with multiple interacting components, such as neural networks or protein interactions. Although a small change of the classical SK model, MSK model already exhibits some interesting and unique features, especially in the non-convex case, the classical Parisi formula~\cite{Tal06,Pan13} for computing the limiting free energy still remains unknown in the non-convex MSK case. We first state the definition of the MSK model as follows.

\subsection{MSK model}\label{s1sec:msk}
 For a spin configuration on the $N$-dimensional hyper-cube, $\mvgs =( \gs_1, \gs_2, \ldots, \gs_N) \in \gS_N := \{-1,+1\}^N$, the Hamiltonian of MSK model is given by 
\begin{align}\label{hamilton}
	H_N(\mvgs) := \frac{\gb}{\sqrt{N}}\sum_{1\le i<j\le N} g_{i,j} \gs_i \gs_j+ h \sum_{i=1}^N \gs_i
\end{align}
where $\{g_{i,j}\}$, the disorder interaction parameters given as independent centered Gaussian random variables, $\gb>0$ is the inverse temperature and $h \ge 0$ is the external field. In the classical SK model, the variance structure of disorder $g_{i,j}$ is homogeneous, usually taken as $g_{i,j} \sim \N(0,1)$ i.i.d. While in the MSK model, the variance of $g_{i,j}$ depends on the structure of species among the $N$ spins.

We use $\sS$ to denote the set of species or types. Assume that, there are $\abs{\sS} =m \ge  2$ species. It is clear that in the case $\abs{\sS}=1$, MSK model reduces to the classical SK model. We partition the set of spins into $m$ disjoint sets, namely,
\[
	I = \bigcup_{s\in \sS} I_s = \{1,2,\ldots, N\},\qquad I_{s}\cap I_{t}=\emptyset\text{ for } s\neq t.
\]
For $i \in I_s, j \in I_t$, we assume
\[ \E g_{i,j}^2 = \gD_{s,t}^2 ,\]
the non-homogeneous interactions of MSK model is encoded in the \emph{variance profile matrix} $\mathbf{\gD^2}:=(\gD_{s,t}^2)_{s,t=1}^m$. Besides that, we assume that the ratio of spins in each species is fixed asymptotically, \ie~for $s\in \sS$ and
\[ \gl_{s,N}:={|I_s|}/{N},\]
we have
\[ \lim_{N\to \infty} \gl_{s,N} = \gl_s \in (0,1). \]
Since $\gl_{s,N}$ and $\gl_s$ are asymptotically the same, for the rest of the article we will use $\gl_s$ instead of $\gl_{s,N}$ for convenience. We denote $\gL:=\diag(\gl_1,\gl_2,\ldots, \gl_m)$.

The overlap vector between two replicas $\mvgs^1, \mvgs^2 \in \gS_N$ is given by
\[ \vR_{12} = (R_{12}^{(1)},R_{12}^{(2)},\ldots,R_{12}^{(m)})^{\T} ,\]
where $$R_{12}^{(s)} := \frac{1}{|I_s|} \sum_{i \in I_s} \gs_i^1 \gs_i^2.$$ is the overlap restricted to species $s \in \{1, \ldots, m\}$. In some cases we write it as $R^{(s)}$ for short if there are only 2 replicas involved. All vectors will be considered as a column vector in the rest of the article.

A central question in spin glass theory is to understand the free energy
\begin{align}
	F_N(\gb,h) := \frac{1}{N} \log Z_N(\gb,h), \quad \text{where} \quad Z_N(\gb,h) := \sum_{\mvgs \in \gS_N} \exp(H_N(\mvgs))
\end{align}
is the partition function, and the associated Gibbs measure is given by
\begin{align}
	G_{N}(\mvgs) = {\exp(H_N(\mvgs))} \cdot Z_N(\gb,h)^{-1},\quad \text{ for $\mvgs\in \gS_{N}$}.
\end{align}
Later we will also use $\la \cdot \ra$ to denote the Gibbs mean just for convenience. 

Under the assumption $\mathbf{\gD}^2$ is \pd\footnote{Some literature also refers the MSK model with positive semi-definite $\mathbf{\gD}^2$ as convex model. If $\mathbf{\gD}^2$ is indefinite, it is also known as non-convex MSK model.}, the Parisi formula for the limiting free energy at all $(\gb,h)$ was established in~\cite{Pan15}. However, it remains a notable challenge to compute the limiting free energy for MSK model with general {\it indefinite} $\mathbf{\gD}^2$. Some minimax variational form was conjectured in~\cite{BGG11}, and this variational form was established in some special cases of MSK model, such as deep Boltzman machine~\cite{Gen23}. Some progresses have been made on the fluctuation results~\cite{DW21a} and replica symmetry breaking in low temperature regime~\cite{BSS19} for MSK model. In the spherical version of multi-species model, Bates and Sohn~\cite{BS22a,BS22b} established the Crisanti-Sommers formula in the multi-species spherical mixed $p$-spin model under the \pd~assumption and further identify a sufficient condition for the sychronization of replica symmetry breaking. Shortly after that, Subag~\cite{Sub23,Subag21T} developed the Thouless-Anderson-Palmer (TAP) approach for multi-species spherical pure $p$-spin model and was able to compute the limiting free energy for both {\it indefinite} and \pd~multi-species pure $p$-spin models. Besides that, McKenna~\cite{Mc23} computed the complexity of the bipartite model, as a special case of non-convex multi-species model, and later Kivimae~\cite{kiv22} further proved the concentration of complexity and use it compute the ground state energy. 

\subsection{TAP equations}
 We briefly discuss the related works on the TAP equations in spin glass theory. TAP equations are a system of consistent equations that was first derived by Thouless, Anderson and Palmer~\cite{TAP77} in 1977. It turns out that the solutions of those equations include fruitful information about the spin glass model, and it was originally derived to understand the free energy of SK model. Later those equations were proved mathematically rigorous in several different approaches by Talagrand~\cite{Tal11a}, Chatterjee~\cite{Cha10} and Adhikari et.al ~\cite{ABVY21} in the high temperature regime. Recently this was also generalized to the low temperature regime~\cite{AJ19} by utilizing the ultrametric structure of Gibbs measure. In~\cite{Bol14}, Bolthausen initiated a TAP iteration scheme to solve the TAP equations. This iteration scheme is also fundamental related to the Approximate message passing algorithms~\cite{JM13,BM11}. There are extensive amount of works concerning the TAP equations, we will not give a comprehensive review of all the works and refer interested readers to~\cite{CT22,CT21} and references therein instead. Due to the fundamental role of TAP equations, in this note, we aim to investigate the Thouless-Anderson-Palmer (TAP) equation for general (both convex and nonconvex) MSK model.

\subsection{Main results}

One key ingredient in the proof of the main results is overlap concentration. In MSK model, one has to deal with overlap vector due to the inhomogeneous species structure. In particular, it is expected that 
\[
R_{12}^{(s)} \approx q_s, \quad \text{for $s \in \sS$}. 
\]
where $\mvq = (q_1,\ldots, q_m)$ is the solution to the following system of equations, 
\begin{align}\label{syseq}
		q_s = \E \tanh^2(\gb \eta \sqrt{(\gD^2\gL \mvq)_s}+h),\qquad s=1,2,\ldots, m,
	\end{align}
where $\eta\sim N(0,1)$. The uniqueness of the solution in some high temperature regime was proved in~\cite{DW21a}. We first recall the critical temperature $\gb_c$ in~\cite{DW21a}. 
Define
\begin{align}\label{betac}
	\gb_{c}:=\rho(\gD^2 \gL)^{-\half}.
\end{align}
where $\rho(A)$ is the \emph{spectral radius} or the largest absolute value of the eigenvalues of $A$. In general, one has $\rho(A) \le \norm{A}$. But it is easy to check that for symmetric $A$, $\rho(A)=\norm{A}$.

The overlap concentration results are formally stated in the following theorem.

\begin{thm}[\cite{DW21a}*{Theorem 1.7}]\label{thm1}
	Assume that $\gb < \gb_0:={\gb_c}/{\sqrt{4\ga}}$, where $\ga =\ga (\gD^{2}):= 1+\ind\{\gD^2 \text{ is indefinite}\}$. For $2\gamma<\gb_c^2-4\ga \gb^2 $, we have
	\begin{align*}
		\gn(\exp(\gamma N\cP(\vR_{12} -\mvq)))\le \det(I- (2\gamma+4\ga \gb^2 ) \cV )^{-1/2}
	\end{align*}
	where
	\begin{align*}
		\cP(\vx):= \vx^{\T}\gL^{\half}\cV\gL^{\half}\vx \quad \text{with} \ \cV:=\abs{\gL^{\half}\gD^2\gL^{\half}}.
	\end{align*}
\end{thm}

This theorem establishes the exponential concentration of overlap vector $\vR_{12}$ on $\mvq$, which satisfies the system of fixed point equations~\eqref{syseq}. It holds for both {\it positive semi-definite} and {\it indefinite} $\mathbf{\gD}$. Now we state our main results for the TAP equation of MSK model.

\begin{thm}\label{thm:main}
For $\gb < \gb_0$, any external field $h$, and integer $k\ge 1$. For any species $s\in \sS$ and $i\in I_s$, we have
\begin{align}
\E\left(\la \gs_i\ra - \tanh\left( \frac{\gb}{\sqrt{N}}\sum_{t=1}^m\sum_{j \in I_t, j \neq i} g_{i,j} \la \gs_j \ra + h - \gb^2(\gD^2\gL(\mv1 - \mvq))_s \la \gs_i \ra \right) \right)^{2k} \le \frac{C}{N^{k}}
\end{align}
where the constant $C$ depends on $k,\gb, \gD^2,\gL$ but not on $N$.
\end{thm}

In the later sections, we use $C$ to denote any finite constant independent $N$ if it does not cause confusion. Compared to the single species case, the Onsager correction term $\gb^2(\gD^2\gL(\mv1 - \mvq))_s \la \gs_i \ra$ in the TAP equations of MSK model suggests an explicit species-wise structure. Expanding the correction term, 
\[
\gb^2(\gD^2\gL (\mv1-\mvq))_s = \gb^2\sum_{t\in \sS} \gl_t \gD^2_{st}(1-q_t),
\]
Alternatively, we can interpret that the spin $\gs_i$ from species $s$ has probability $\gl_t$ interacting with spins in species $t$, and $\gD^2_{st}$ is the corresponding renormalization. The total effects from all species on $\gs_i$ is just average weighted corrections, which is exactly the r.h.s. We illustrate this in the following particular examples of multi-species models. 

\begin{ex}[Multiple copies of single species model]
Let us consider two copies of single species model, which is a convex MSK model. Take 
\[
\gD^2 = 
\begin{pmatrix}
\gD^2_{1,1} & 0 \\
0 & \gD^2_{2,2}
\end{pmatrix}
\quad 
\text{and}
\quad 
\gL = 
\begin{pmatrix}
\gl_1 & 0 \\
0 & \gl_2
\end{pmatrix}.
\]
where $\gl_1+\gl_2 =1$. In this case, for $i\in I_1$, the TAP equation is 
\[
\langle \gs_i \rangle \approx \tanh \left( \frac{\gb}{\sqrt{N}} \sum_{j\in I_1, j\neq i} g_{i,j} \langle \gs_j \rangle + h - \gb^2 \gl_1 \gD^2_{1,1}(1-q_1) \langle \gs_i \rangle \right)
\]
the Onsager correction term in this case just contains the effects inside species 1 itself, since in this case the model is just two copies of single species model.
\end{ex}

\begin{ex}[Bipartite model]
Now if we remove the interactions inside species and only allow inter-species interactions, this becomes a bipartite model. Take 
\[
\gD^2 = 
\begin{pmatrix}
0 &  \gD^2_{1,2}\\
\gD^2_{1,2} & 0
\end{pmatrix}
\quad 
\text{and}
\quad 
\gL = 
\begin{pmatrix}
\gl_1 & 0 \\
0 & \gl_2
\end{pmatrix}.
\]
In this case, for $i\in I_1$, the TAP equation is 
\[
\langle \gs_i \rangle \approx \tanh \left( \frac{\gb}{\sqrt{N}} \sum_{j\in I_2, j\neq i} g_{i,j} \langle \gs_j \rangle + h - \gb^2 \gl_2 \gD^2_{1,2}(1-q_2) \langle \gs_i \rangle \right)
\]
the Onsager correction term in this case just contains the inter-species effects from species 2.
\end{ex}

\section{Proof of main results}
We present the proof of Theorem~\ref{thm:main} in this section. The first key ingredient is the analog of Theorem 1.7.11 in~\cite{Tal11a}. We state it formally as follows.

\begin{thm}\label{thm:key}
Assume that $\gb< \gb_0$. Consider an infinitely differentiable function $U \in \cC^{\infty}(\bR)$, we further assume for all $\ell, k \in \bN^+$, and for any Gaussian random variable $z$, 
\begin{align}\label{cond:der}
\E \abs{U^{(\ell)}(z) }^k < \infty.
\end{align}
For fixed $s \in \sS$, consider independent Gaussian r.v.s $\{\eta_j^s\}_{j=1}^N$ with mean 0 and variance $\E (\eta_j^s)^2 = \gD_{s,t}^2$ for $j\in I_t$, and $\xi$ standard Gaussian, and they are all independent from the disorder $g_{i,j}$ in the Gibbs measure, then for each $k \in \bN^+$, we have 
\begin{align}\label{eq:main}
\E \left( \left\la U \left(\frac{1}{\sqrt{N}} \sum_{j \le N} \eta_j^s \bar{\gs}_i\right) \right\ra - \E_{\xi} U\left(\xi \sqrt{\gD^2\gL(\mv1 - \mvq)_s} \right) \right)^{2k} \le \frac{C}{N^{k}},
\end{align}
where $\bar{\gs}_i :=\gs_i -\la \gs_i\ra$ and $\mvq$ is the unique solution to the equation~\eqref{syseq}, and the constant $C$ depends on $\gL, \gD^2, \gb, k, U$ but not on $N$.
\end{thm}

The proof relies on the concentration results in Theorem~\ref{thm:key}. Based on this lemma,  we have the following two useful Corollaries.  

\begin{cor}\label{cor:1}
Assume $\gb < \gb_0$, for any external field $h$ and fixed $s \in \sS$, let $\eps_s \in \{+1,-1\}$ and $k \ge 1$, we have 
\begin{align}
\E \left( \left \la \exp \frac{\eps_s \gb}{\sqrt{N}} \sum_{j \le N} \eta_j^s \gs_j \right \ra - \exp \frac{\gb^2(\gD^2 \gL(\mv1 - \mvq))_s}{2}  \cdot \exp \frac{\eps_s \gb}{\sqrt{N}} \sum_{j \le N} \eta_j^s \la \gs_j \ra \right)^{2k} \le \frac{C}{N^k},
\end{align}
and 
\begin{align}
& \E \Bigg( \left \la  \frac{1}{\sqrt{N}} \sum_{j \le N} \eta_j^s \bar{\gs}_i \exp \frac{\eps_s \gb}{\sqrt{N}} \sum_{j \le N} \eta_j^s \gs_j \right \ra   \\
&\qquad - \eps_s \gb (\gD^2\gL(\mv1 - \mvq))_s \cdot \exp \frac{\gb^2(\gD^2\gL(\mv1 - \mvq))_s}{2} \cdot  \exp \frac{\eps_s \gb}{\sqrt{N}} \sum_{j \le N} \eta_j^s \la \gs_j \ra \Bigg) \le \frac{C}{N^k},
\end{align}
where $C$ is some constant independent of $N$.
\end{cor}

\begin{proof}
By applying Theorem~\ref{thm:key} with $U(x) = \exp(\eps_s \gb x)$, taking $x = \frac{1}{\sqrt{N}} \sum_{j \le N} \eta_j^s \bar{\gs}_j$, we have 
\[
\E \left(\left\la  \exp \frac{\eps_s \gb}{\sqrt{N}} \sum_{j \le N} \eta_j^s \bar{\gs}_j - \exp\frac{\gb^2}{2}(\gD^2\gL(\mv1 - \mvq))_s \right \ra\right)^{4k} \le \frac{C}{N^{2k}},
\]
On the other hand, by standard Gaussiain estimates, notice that 
\[
\E \left( \exp \frac{\eps_s \gb}{\sqrt{N}} \sum_{j \le N}\eta_j^s \la \gs_j \ra \right)^{4k} \le C^k,
\]
recall $C$ is some finite constant independent of $N$.
By Cauchy-Schwartz, we have 
\begin{align*}
& \E \left( \left \la \exp \frac{\eps_s \gb}{\sqrt{N}} \sum_{j \le N} \eta_j^s \gs_j \right \ra - \exp \frac{\gb^2(\gD^2 \gL(\mv1 - \mvq))_s}{2}  \cdot \exp \frac{\eps_s \gb}{\sqrt{N}} \sum_{j \le N} \eta_j^s \la \gs_j \ra \right)^{2k} \\
\le & \left( \E \left(\left\la  \exp \frac{\eps_s \gb}{\sqrt{N}} \sum_{j \le N} \eta_j^s \bar{\gs}_j - \exp\frac{\gb^2}{2}(\gD^2\gL(\mv1 - \mvq))_s \right \ra\right)^{4k} \cdot \E \left( \exp \frac{\eps_s \gb}{\sqrt{N}} \sum_{j \le N}\eta_j^s \la \gs_j \ra \right)^{4k} \right)^{1/2} \\
\le & \frac{C}{N^k}.
\end{align*}
The proof for the second inequality can be completed in a similar fashion. Applying Theorem~\ref{thm:key} with $U(x) = x \exp (\eps_s \gb x)$ for $x = \frac{1}{\sqrt{N}} \sum_{j \le N} \eta_j^s \bar{\gs}_j$, we have
\begin{align*}
 \E \Bigg( \left \la  \frac{1}{\sqrt{N}} \sum_{j \le N} \eta_j^s \bar{\gs}_i \exp \frac{\eps_s \gb}{\sqrt{N}} \sum_{j \le N} \eta_j^s \gs_j \right \ra   
 - \eps_s \gb (\gD^2\gL(\mv1 - \mvq))_s \cdot \exp \frac{\gb^2(\gD^2\gL(\mv1 - \mvq))_s}{2}  \Bigg) \le \frac{C}{N^k},
\end{align*}
where we used Gaussian integration by parts 
\[
\E \xi \exp  \left(\xi \cdot  \eps_s \gb\sqrt{(\gD^2\gL(\mv1 - \mvq))_s} \right) = \eps_s \gb(\gD^2\gL(\mv1 - \mvq))_s \exp \frac{\gb^2(\gD^2\gL(\mv1 - \mvq))_s}{2}.
\]
The rest follows similarly by Cauchy-Schwarz.
\end{proof}

\begin{cor}\label{cor:2}
Let $\cE_s : = \exp \left( \frac{\eps_s \gb}{\sqrt{N}} \sum_{j \le N} \eta_j^s \gs_j + \eps_s h\right)$, for each $s \in \sS$, we have 

\begin{align}\label{eq:cor2-1}
\E \left( \frac{\la \av \eps_s \cE_s \ra}{\la \av \cE_s \ra} - \tanh\left(  \frac{\gb}{\sqrt{N}} \sum_{j \le N} \eta_j^s \la \gs_j \ra + h\right)\right)^{2k} \le \frac{C}{N^k}, 
\end{align}
and 
\begin{align}\label{eq:cor2-2}
\E \left(  \frac{1}{\sqrt{N}} \sum_{j \le N} \eta_j^s \frac{\la \gs_j \av \cE_s \ra }{\la \av \cE_s \ra} - \gb (\gD^2\gL(\mv1 - \mvq))_s \frac{\la \av \eps_s \cE_s \ra}{\la \av \cE_s \ra} - \frac{1}{\sqrt{N}} \sum_{j \le N} \eta_j^s \la \gs_j \ra \right)^{2k} \le \frac{C}{N^k}
\end{align}
where the operator $\av$ is the average over $\eps_s = \pm 1$, again the constant $C$ is independent of $N$.
 \end{cor}

\begin{proof}
We start proving the first inequality. To do that, we first derive approximations for $\la \av \eps_s \cE_s \ra $ and $\la \av \cE_s \ra $.

By the definition of $\av$, we have 
\begin{align*}
 \la \av \eps_s \cE_s \ra  = \left\la \frac12 \cdot\exp \left( \frac{ \gb}{\sqrt{N}} \sum_{j \le N} \eta_j^s \gs_j +  h\right) - \frac12 \cdot \exp \left( \frac{- \gb}{\sqrt{N}} \sum_{j \le N} \eta_j^s \gs_j -h\right) \right\ra,
\end{align*}
By Corollary~\ref{cor:1}, we know that 
\begin{align*}
& \E \Bigg( \left \la \exp \eps_s\left(\frac{ \gb}{\sqrt{N}} \sum_{j \le N} \eta_j^s \gs_j +h \right) \right \ra  \\
 &\qquad \qquad - \exp \frac{\gb^2(\gD^2 \gL(\mv1 - \mvq))_s}{2}  \cdot \exp \eps_s \left(\frac{ \gb}{\sqrt{N}} \sum_{j \le N} \eta_j^s \la \gs_j \ra +  h \right)\Bigg)^{2k} \le \frac{C}{N^k},
\end{align*}
This implies 
\begin{align}\label{eq:approx1}
\E \left( \la \av \eps_s \cE_s \ra - \exp \frac{\gb^2(\gD^2 \gL(\mv1 - \mvq))_s}{2}  \cdot \sinh \left( \frac{ \gb}{\sqrt{N}} \sum_{j \le N} \eta_j^s \la \gs_j \ra + h\right) \right) \le \frac{C}{N^k}.
\end{align}
Similarly, we have 
\[
\E \left( \la \av  \cE_s \ra - \exp \frac{\gb^2(\gD^2 \gL(\mv1 - \mvq))_s}{2}  \cdot \cosh \left( \frac{ \gb}{\sqrt{N}} \sum_{j \le N} \eta_j^s \la \gs_j \ra + h\right) \right) \le \frac{C}{N^k}.
\]
Together it gives rise to the first inequality~\eqref{eq:cor2-1} combining with the following fact: 
\[
\abs{\frac{A'}{B'} - \frac{A}{B}}\le \abs{A-A'}+ \abs{B-B'} \quad \text{for} \ \abs{A'}\le B', B\ge1.
\]
To prove the second inequality, it only needs to prove the following fact: 
\begin{align}\label{eq:fact}
\frac{1}{\sqrt{N}} \sum_{j \le N} \eta_j^s \left(\la  \gs_j   \av \cE_s \ra - \la \gs_j\ra \la \av \cE_s\ra  \right)\approx \gb (\gD^2\gL(\mv1-\mvq))_s \la \av \eps_s \cE_s \ra.
\end{align}
Notice that the l.h.s is just $\frac{1}{\sqrt{N}} \sum_{j \le N } \eta_j^s \la \bar{\gs}_j \av \cE_s \ra $, by the definition of $\av$ and $\cE_s$, applying the second inequality in Corollary~\ref{cor:1}, it leads to 
\begin{align*}
& \E \Bigg( 
 \gb^2(\gD^2\gL(\mv1-\mvq))_s \cdot \exp \frac{\gb^2 (\gD^2\gL(\mv1-\mvq))_s}{2}\cdot \sinh\left( \frac{\gb}{\sqrt{N}} \sum_{j\le N}\eta_j^s \la \gs_j \ra + h \right) \\
 & \qquad \quad -\frac{1}{\sqrt{N}} \sum_{j \le N } \eta_j^s \la \bar{\gs}_j \av \cE_s \ra \Bigg)^{2k} \le \frac{C}{N^k}.
\end{align*}
Combining with the results for $\la \av \eps_s \cE_s \ra$ in~\eqref{eq:approx1}, we have the desired result in~\eqref{eq:fact}. 
\end{proof}

Now we present the proof for the key result in Theorem~\ref{thm:key}. Before that, let us recall the following useful lemma based on Gaussian integration by parts.

\begin{lem}[\cite{Tal11b}*{Lemma 1.3.1}]\label{lem:derivative}
Consider two independent centered Gaussian vectors $\mvu = \{u_i\}_{i\in I}$ and $\mvv = \{v_i\}_{i\in I}$ with index set $I$. The interpolation $\mvu^r= \{u^r_i\}_{i\in I}$ is given by $u^r_i:= \sqrt{r} \cdot u_i + \sqrt{1-r} \cdot v_i$ for $r\in[0,1]$. For $F \in C^{\infty}(\bR^{\abs{I}})$, the derivative of $\phi(r):=\E F(\mvu^r)$ is given by 
\[
\phi'(r) = \frac{1}{2}\sum_{i,j\in I}(\E u_iu_j -\E v_i v_j) \E \left(\frac{\partial^2 F}{\partial x_i \partial x_j }(\mvu^r) \right).
\]

\end{lem}

Now we turn to the proof of Theorem~\ref{thm:key}, which is based on the overlap concentration results and the Lemma~\ref{lem:derivative}.

\begin{proof}[Proof of Theorem~\ref{thm:key}]

We use the smart path interpolation to prove the theorem. For convenience, let 
\[
V(x) = U(x) - \E_{\xi} U\left(\xi \sqrt{(\gD^2\gL(\mv1-\mvq))_s}\right),
\]
such that $\E_{\xi} V(\xi \sqrt{(\gD^2\gL(\mv1-\mvq))_s}= 0$, and for $\ell \le 2k$,
\[
S^\ell_s = \frac{1}{\sqrt{N}} \sum_{j\le N } \eta_j^s \bar{\gs}_j^{\ell}, \ \text{where} \ \mvgs^\ell \in \{ -1,+1\}^N.
\]

 For $r \in [0,1]$ and $\{\xi^\ell\}_{\ell \le 2k}$ independent standard Gaussians, consider the following interpolated Gaussian fields,
\[
S_s^\ell(r) := \sqrt{r}\cdot  S^\ell_s + \sqrt{1-r}\cdot \xi^\ell \sqrt{(\gD^2\gL(\mv1-\mvq))_s}.
\]
The goal is to control $\E \left\la \prod_{\ell \le 2k} V(S_s^{\ell}(1) \right \ra $, which is the l.h.s of~\eqref{eq:main}. In general, we let
\[
\phi(r) := \E \left\la \prod_{\ell \le 2k} V(S_s^{\ell}(r)) \right \ra = \E \left\la \E_\xi\prod_{\ell \le 2k}  V(S_s^{\ell}(r)) \right \ra,
\]
where the second equality just splits the total expectation. 
To prove~\eqref{eq:main}, we need to prove the following derivative bounds
\begin{align*}
\frac{\partial^p \phi}{\partial r^p}(r)\bigg\lvert_{r=0} = 0 \ \text{for} \ p<2k; \ \text{and} \ \abs{\frac{\partial^{2k} \phi}{\partial r^{2k}}(r)} \le \frac{C}{N^k}.
\end{align*}
Next, let us turn to compute those derivatives by using the Lemma~\ref{lem:derivative}. Taking $F(\vx):= \prod_{\ell \le 2k} V(x_\ell)$ for $\vx:=(x_1,x_2,\ldots, x_{2k})$. Notice that $\phi(r)=\E \la F(\vS_s(r))\ra$ with $\vS_s(r):=(S^1_s(r),S^2_s(r),\ldots, S^{2k}_s(r))$, applying the derivative formula iteratively for $p$ times, we have 
\begin{align}\label{eq:deriv}
\frac{\partial^p \phi}{\partial r^p}(r) = 2^{-p}  \sum_{\ell_1, \ell_1', \cdots, \ell_p, \ell_p'} \E \left \la  \prod_{j=1}^pT(\ell_j,\ell_j') \frac{\partial^{\mvo} F }{\partial \mvx^{\mvo}}\right \ra,
\end{align}
where $\mvo=\{ o_\ell \}_{\ell \le 2k}$ corresponds to the order of derivative with respect to $\{x_{\ell}\}_{\ell \le 2k}$. In particular, 
\[
\frac{\partial^{\mvo} F }{\partial \mvx^{\mvo}} = \prod_{\ell \le 2k} V^{(o_\ell)}(x_\ell).
\]
Besides, for $\ell = \ell'$,
\begin{align*}
T(\ell , \ell) & :=
 \E[ S^{\ell}_s]^2 - \E [\xi^{\ell} \sqrt{\gD^2\gL(\mv1-\mvq)_s}]^2 \\
 & =\frac{1}{N}\sum_{t \in \sS} \gD_{s,t}^2 \sum_{j\in I_t}(\bar{\gs}_i^\ell)^2 - \gD^2\gL(\mv1-\mvq)_s, 
\end{align*}
and for $\ell \neq \ell'$, by the independence property of $\xi^{\ell}, \xi^{\ell'}$,

\begin{align*}
T(\ell , \ell') & := \E [S_s^\ell S_s^{\ell'}] - \E [\xi^\ell \xi^{\ell'} \sqrt{\gD^2\gL(\mv1-\mvq)_s}] \\
& = \frac{1}{N}\sum_{t \in \sS} \gD_{s,t}^2 \sum_{ j\in I_t}  \bar{\gs}_i^{\ell}\bar{\gs}_i^{\ell'}.
\end{align*}

Let us first verify $\frac{\partial^p \phi}{\partial r^p}(r)\big\lvert_{r=0} = 0 \ \text{for} \ p<2k$. By~\eqref{eq:deriv}, 
\begin{align*}
\frac{\partial^p \phi}{\partial r^p}(r)\bigg\lvert_{r=0} & = 2^{-p}  \sum_{\ell_1, \ell_1', \cdots, \ell_p, \ell_p'} \E \frac{\partial^{\mvo} F }{\partial \mvx^{\mvo}}\left( \left(\xi^\ell\sqrt{(\gD^2\gL(\mv1-\mvq))_s}\right)_{\ell \le 2k}\right) \E \left \la  \prod_{j=1}^pT(\ell_j,\ell_j') \right \ra, \\
& = 2^{-p}  \sum_{\ell_1, \ell_1', \cdots, \ell_p, \ell_p'} \E \prod_{\ell \le 2k} V^{(o_\ell)}\left( \left(\xi^\ell\sqrt{(\gD^2\gL(\mv1-\mvq))_s}\right)_{\ell \le 2k}\right) \E \left \la  \prod_{j=1}^pT(\ell_j,\ell_j') \right \ra,
\end{align*}
By the fact $\E V(\xi (\gD^2\gL(\mv1-\mvq))_s) = 0 $, we know if $o_{\ell} \ge 1 \ \text{for all $\ell\le 2k$}$, then 
\[
 \E \prod_{\ell \le 2k} V^{(o_\ell)}\left( \left(\xi_\ell\sqrt{(\gD^2\gL(\mv1-\mvq))_s}\right)_{\ell \le 2k}\right) \neq 0,
\]
otherwise it will be 0. It means that every number $\ell \le 2k$ appears at least once in $\ell_1, \ell_1', \ldots, \ell_p,\ell_p'$. On the other hand, for $\ell \neq \ell'$, notice that the average of $T(\ell,\ell')$ over $\mvgs^\ell$ and $\mvgs^{\ell'}$ are both zero. This implies that in order for $\left\la \prod_{j=1}^p T(\ell_j,\ell_j') \right \ra \neq 0 $, it needs every number $\ell \le 2k$ must appear at least twice in the list $\ell_1, \ell_1', \ldots, \ell_p,\ell_p'$. It further implies the total length of the list $2p \ge 4k$, \ie $p \ge 2k$. Otherwise $\frac{\partial^p \phi}{\partial r^p}(r)\big\lvert_{r=0}=0$ for $p< 2k$.

Next we verify the second claim, it needs a control of $T(\ell,\ell')$. Notice that $T(\ell, \ell')$ is intimately connected with the centered overlaps. In particular, recall $\bar{\gs}_i^\ell  = \gs_i^\ell - \la \gs_i^\ell \ra $, then 

\begin{align*}
T(\ell, \ell)  & = \frac{1}{N}\sum_{t \in \sS} \gD_{s,t}^2 \sum_{j\in I_t}(\bar{\gs}_i^\ell)^2 - \gD^2\gL(\mv1-\mvq)_s \\
& = -2 \left( \frac1N \sum_{t \in \sS} \gD^2_{s,t} \sum_{i \in I_t} ( \gs_i^\ell \la \gs_i^\ell \ra -q_t) \right)  + \frac{1}{N}\sum_{t \in \sS} \gD_{s,t}^2\sum_{i \in I_t} (\la \gs_i^\ell \ra^2 -q_t).
\end{align*}

For the first term on the r.h.s, using the fact $\la \gs_i^\ell \ra = \la \gs_i^{\ell'}\ra$ and by Jensen's inequality with respect to $\la \gs_i^{\ell'} \ra$, we have 
\begin{align*}
\E \left \la \left(  \frac1N \sum_{t \in \sS} \gD^2_{s,t} \sum_{i \in I_t} ( \gs_i^\ell \la \gs_i^\ell \ra -q_t) \right)^{2p} \right \ra  &  = \E \left \la \left(  \frac1N \sum_{t \in \sS} \gD^2_{s,t} \sum_{i \in I_t} ( \gs_i^\ell \la \gs_i^{\ell'} \ra -q_t) \right)^{2p} \right \ra \\
& \le \E \left\la \left(\sum_{t\in \sS} \gD^2_{s,t} \gl_t \bar{R}_{\ell,\ell'}^{(t)} \right)^{2p} \right\ra \le C\cdot N^{-p}, 
\end{align*}
the last step is based on the exponential overlap concentration in Theorem~\ref{thm1}. 
Similarly, for the second term on the r.h.s, we have 
\[
\frac{1}{N}\sum_{t \in \sS} \gD_{s,t}^2\sum_{i \in I_t} (\la \gs_i^\ell \ra^2 -q_t) = \frac{1}{N}\sum_{t \in \sS} \gD_{s,t}^2\sum_{i \in I_t} (\la \gs_i^\ell \ra \la \gs_{i}^{\ell'}\ra  -q_t),
\]
then applying Jensen's inequality under $\la \cdot \ra$, it can be related to the overlap concentration again. On the other hand, for $\ell \neq \ell'$, $T(\ell,\ell')$ can be rewritten as 
\begin{align*}
 T(\ell,\ell') & = \frac{1}{N}\sum_{t \in \sS} \gD_{s,t}^2 \sum_{ j\in I_t}  ({\gs}_i^{\ell}{\gs}_i^{\ell'} -q_t)- \frac{1}{N}\sum_{t \in \sS} \gD_{s,t}^2 \sum_{ j\in I_t}  ({\gs}_i^{\ell}\la {\gs}_i^{\ell'} \ra -q_t) \\
 & \quad - \frac{1}{N}\sum_{t \in \sS} \gD_{s,t}^2 \sum_{ j\in I_t}  (\la {\gs}_i^{\ell}\ra {\gs}_i^{\ell'} -q_t) +  \frac{1}{N}\sum_{t \in \sS} \gD_{s,t}^2 \sum_{ j\in I_t}  (\la {\gs}_i^{\ell} \ra \la {\gs}_i^{\ell'}\ra  -q_t),
\end{align*}
all the terms on the r.h.s can be similarly reduced to function of overlaps and upper bounded accordingly by Theorem~\ref{thm1}. Putting all these together, we have 
\[
\forall p\ge1, \quad \E \la T(\ell,\ell')^{2p}\ra  \le C \cdot N^{-p} \quad \text{for all $\ell,\ell'$}. 
\]
With this input, combining with the computation of derivative in~\eqref{eq:deriv} and the condition~\ref{cond:der}, by H\"older inequality, we can easily verify that 
\[
\abs{\frac{\partial^{2k} \phi}{\partial r^{2k}}(r)} \le \frac{C}{N^k}.
\]

\end{proof}

\subsection{Proof of Theorem~\ref{thm:main}}

We start with the cavity idea, fixing a species $s\in \sS$, suppose we remove a spin $\gs_i$ in that species~\ie~$i\in I_s$ to create cavity, the Hamiltonian of the associated size $N-1$ system is 
\[
H_{N-1}(\mvgs_{-s}) = \frac{\gb}{\sqrt{N}} \sum_{k<j; k,j \neq i} g_{k,j} \gs_k \gs_j + h \sum_{k\neq i} \gs_k,
\]
where $\mvgs_{-s} \in \{ -1,+1\}^{N-1}$ by removing the $i$-th coordinate in $\mvgs \in \gS_N$. 
Note that $\gb_{-}:= \gb \cdot \sqrt{\frac{N-1}{N}} \le \gb$, by \cite{Tal11a}*{Proposition 1.6.1}, it's easy to see for $i \in I_s$
\[
\la \gs_i \ra  = \frac{\la \av \eps_s \tilde{\cE}_s \ra_{s-}}{\la \av \tilde{\cE}_s \ra_{s-}},
\]
where the Gibbs average $\langle \cdot \rangle_{s-}$ is w.r.t the Hamiltonian $H_{N-1}(\mvgs_{-s})$. The operator $\text{AV}$ is just taking average w.r.t the spin removed, that is $\eps_s$, and
\[
\tilde{\cE}_s := \exp \left(\frac{\eps_s \gb}{\sqrt{N}} \sum_{j \neq i} g_{i,j} \gs_j + \eps_s h \right) = \exp \left(\frac{\eps_s \gb_{-}}{\sqrt{N-1}} \sum_{j \neq i} g_{i,j} \gs_j + \eps_s h \right).
\]
 From the above expression, we have $\frac{\gb}{\sqrt{N}} = \frac{\gb_-}{\sqrt{N-1}}$. The cavity argument results in a small change of $\gb$ to $\gb_-$, this change will also create a small shift on $\mvq$ in the fixed point equation~\eqref{syseq}. We define the shifted fixed point as $\mvq_{s-}$
\begin{align}
    (\mvq_{s-})_t = \E \tanh^2(\gb_- \eta \sqrt{(\gD^2\gL \mvq_{s-})_t}+h),\quad t=1,2,\ldots, m.
\end{align}
The following lemma formally states the difference between $\mvq$ and $\mvq_{s-}$ is small.
\begin{lem}\label{lem:close-q}
    For $\gb<\gb_0$, we have 
    \[
    \norm{\mvq - \mvq_{s-}}_1 \le \frac{C}{N}
    \]
\end{lem}

\begin{proof}[Proof of Lemma~\ref{lem:close-q}]
     Consider the function $f_t(\gb, \mvq):=\E \tanh^2(\gb\eta \sqrt{(\gD^2\gL \mvq)_t}+h)$ for $t=1,2,\cdots,m$. Let $\mvq(\gb)=(q_1(\gb),q_2(\gb),\cdots, q_m(\gb))$ be defined via the following fixed point equations
     \[
     q_t(\gb) = \E \tanh^2(\gb \eta \sqrt{(\gD^2\gL \mvq(\gb))_t}+h) \quad \text{for} \ t=1,2,\cdots, m.
     \]
     Note that $\mvq_{s-}$ is obtained by perturbing the $\gb$ parameter in the equations, thus in order to bound the difference between $\mvq, \mvq_{s-}$, one just to bound the gradient $\nabla \mvq(\gb)$. It is easy to see that for each $t$,
     \[
     q_t'(\gb) = \frac{\partial f_t(\gb,\mvq(\gb))}{\partial \gb}+ \sum_{r=1}^m \frac{\partial f_t(\gb,\mvq(\gb))}{\partial q_r} \cdot q_r'(\gb).
     \]
     Written in the matrix form,
     \begin{align}
         \begin{pmatrix}
             1-\frac{\partial f_1(\gb,\mvq)}{\partial q_1} & \frac{\partial f_1(\gb,\mvq)}{\partial q_2} & \cdots & \frac{\partial f_1(\gb,\mvq)}{\partial q_m} \\
             \frac{\partial f_2(\gb,\mvq)}{\partial q_1} & 1-\frac{\partial f_2(\gb,\mvq)}{\partial q_2} & \cdots & \frac{\partial f_2(\gb,\mvq)}{\partial q_m} \\
             \vdots & \vdots & \vdots & \vdots \\
             \frac{\partial f_m(\gb,\mvq)}{\partial q_1} & \frac{\partial f_m(\gb,\mvq)}{\partial q_2} & \cdots & 1-\frac{\partial f_1(\gb,\mvq)}{\partial q_m} 
         \end{pmatrix}
         \begin{pmatrix}
             q_1'(\gb) \\
             q_2'(\gb) \\
             \vdots   \\
             q_m'(\gb)
         \end{pmatrix}
         =
         \begin{pmatrix}
             \frac{\partial f_1(\gb,\mvq)}{\partial \gb}\\
             \frac{\partial f_2(\gb,\mvq)}{\partial \gb} \\
             \vdots \\
             \frac{\partial f_m(\gb,\mvq)}{\partial \gb}
         \end{pmatrix}
     \end{align}
     By Gaussian integration by parts, one can compute 
     \[
     \frac{\partial f_t(\gb,\mvq)}{\partial \gb}= \gb (\gD^2 \gL \mvq)_t \E g''(\gb \eta \sqrt{(\gD^2\gL \mvq)_t}+h),
     \]
     \[
     \frac{\partial f_t(\gb,\mvq)}{\partial q_r} = \frac{\gb^2 \gD^2_{t,r}\gl_r}{2} \E g''(\gb \eta \sqrt{(\gD^2\gL \mvq)_t}+h),
     \]
     for $t,r=1,2,\cdots, m$, where we denote $g(x):=\tanh^2(x)$. Note that $g''(x) =( 2-4\sinh^2(x))/\cosh^4(x) \le 2$ and when $\gb<\gb_0$, clearly this implies that $\nabla \mvq(\gb)< \infty$. Using the relation $\gb,\gb_-$, it gives the desired bound on $\abs{\mvq - \mvq_{s-}}$.
     
\end{proof}

Now applying the Corollary~\ref{cor:2} w.r.t the Hamiltonian $H_{N-1}(\mvgs_{-s})$, we have 

\begin{align}\label{eq:key}
\E \left(\la \gs_i \ra  - \tanh\left( \frac{\gb}{\sqrt{N}} \sum_{t=1}^m\sum_{j\in I_t, j \neq i} g_{i,j} \la \gs_j \ra_{s-} + h\right)\right)^{2k} \le \frac{C}{N^k}.
\end{align}
Using the second inequality of Corollary~\ref{cor:2}, we have 
\[
\E \left( \frac{\gb_{-}}{\sqrt{N-1}} \sum_{t=1}^m \sum_{j \in I_t, j\neq i} g_{i,j} \la \gs_j \ra - \gb_{-}^2(\gD^2\gL(\mv1-\mvq_{s-}))_s - \frac{\gb_-}{\sqrt{N-1}} \sum_{t=1}^m\sum_{j\in I_t, j\neq i} g_{i,j} \la\gs_j \ra_{s-}\right)^{2k} \le \frac{C}{N^k}.
\]
By the Lemma~\ref{lem:close-q} for the closeness of $\gb$ and $\gb_{s-}$, $\mvq$ and $\mvq_{s-}$, it gives
\[
 \E\left(\frac{\gb}{\sqrt{N}}\sum_{t=1}^m \sum_{i\in I_t, j \neq i} g_{i,j} \la \gs_j \ra - \gb^2(\gD^2\gL(\mv1-\mvq))_s \la \gs_i \ra - \frac{\gb}{\sqrt{N}} \sum_{t=1}^m\sum_{j\in I_t, j\neq i} g_{i,j} \la\gs_j \ra_{s-} \right)^{2k} \le \frac{C}{N^k}.
\]
Using the following elementary fact, 
\[
\E (\tanh(X) - \tanh(Y))^{2k} \le \E(X- Y)^{2k} \le \frac{C}{N^k}, \ \text{for random variables} \ X, Y.
\]
This gives 
\begin{align*}
 & \E \Bigg( \tanh\bigg(\frac{\gb}{\sqrt{N}} \sum_{t=1}^m\sum_{j\in I_t, j \neq i} g_{i,j} \la \gs_j \ra +h - \gb^2(\gD^2\gL(\mv1-\mvq))_s \la \gs_i \ra \bigg)  \\
  &  \quad \qquad \qquad \qquad - \tanh\bigg(\frac{\gb}{\sqrt{N}} \sum_{t=1}^m\sum_{j\in I_t, j\neq i} g_{i,j} \la\gs_j \ra_{s-} \bigg) \Bigg)^{2k} \le \frac{C}{N^k}.
\end{align*}
Further by~\eqref{eq:key}, we have 
\[
\E\left(\la \gs_i\ra - \tanh\left( \frac{\gb}{\sqrt{N}}\sum_{t=1}^m\sum_{j\in I_t, j \neq i} g_{i,j} \la \gs_j \ra + h - \gb^2(\gD^2\gL(\mv1 - \mvq))_s \la \gs_i \ra \right) \right)^{2k} \le \frac{C}{N^{k}}.
\]

\bibliography{MSKTAP.bib} 

\end{document}